\documentclass[a4paper,11pt] {amsart}
\usepackage{amsthm,amssymb,latexsym,mathrsfs}
\usepackage{color}

\input amssym.def

\author{Beno\^it F. Sehba}

\title[Derivatives characterization of Bergman-Orlicz spaces]{Derivatives characterization of Bergman-Orlicz spaces and applications}
\newtheorem{theorem}{T{\hskip 0pt\footnotesize\bf HEOREM}}[section]
\newtheorem{lemma}[theorem]{L{\hskip 0pt\footnotesize\bf EMMA}}
\newtheorem{proposition}[theorem]{P{\hskip 0pt\footnotesize\bf ROPOSITION}}
\newtheorem{definition}[theorem]{D{\hskip 0pt\footnotesize\bf EFINITION}}
\newtheorem{corollary}[theorem]{C{\hskip 0pt\footnotesize\bf OROLLARY}}

\newcommand{\bprop} {\begin{proposition}}
\newcommand{\eprop} {\end{proposition}}
\newcommand{\btheo} {\begin{theorem}}
\newcommand{\etheo} {\end{theorem}}
\newcommand{\blem} {\begin{lemma}}
\newcommand{\elem} {\end{lemma}}
\newcommand{\bcor} {\begin{corollary}}
\newcommand{\ecor} {\end{corollary}}

\newcommand{\Be}{\begin{equation}}
\newcommand{\Ee}{\end{equation}}
\newcommand{\Bea}{\begin{eqnarray}}
\newcommand{\Eea}{\end{eqnarray}}
\newcommand{\Bes}{\begin{equation*}}
\newcommand{\Ees}{\end{equation*}}
\newcommand{\Beas}{\begin{eqnarray*}}
\newcommand{\Eeas}{\end{eqnarray*}}
\newcommand{\Ba}{\begin{array}}
\newcommand{\Ea}{\end{array}}

\def\C{\mathbb{C}}

\scrollmode

\begin{document}
\date{\today}
\address{Beno\^it Sehba, Department of Mathematics, University of Ghana, Legon\\ P. O. Box LG 62 Legon-Accra, Ghana.}
\email{bfsehba@ug.edu.gh}

\begin{abstract} It is well known that a function is in a Bergman space of the unit ball if and only if it satisfies some Hardy-type inequalities. We extend this fact to Bergman-Orlicz spaces. As applications, we obtain Gustavsson-Peetre interpolation of two Bergman-Orlicz spaces and we completely characterize symbols  of bounded or compact Ces\`aro-type operators on Bergman-Orlicz spaces, extending known results for classical weighted Bergman spaces.
\end{abstract}
\maketitle

\section{Introduction}


Let $z=(z_1,\cdots,z_n)$ and $w=(w_1,\cdots,w_n)$ be vectors in
$\C^n$. We write
$$\langle z,w\rangle =z_1\overline {w_1} +
\cdots + z_n\overline {w_n}$$
and $|z|^2=\langle
z,z\rangle =|z_1|^2 +\cdots +|z_n|^2$.

 Given a function $\Phi:[0,\infty)\rightarrow [0,\infty)$, We say  $\Phi$ is a growth function if it is a continuous and non-decreasing function.

 We denote by $d\nu$ the Lebesgue measure on $\mathbb B^n$ the unit ball  of $\mathbb C^n$, and $d\sigma$  the normalized measure on $\mathbb S^n=\partial{\mathbb B^n}$ the boundary of
$\mathbb B^n$. As usual, we denote by $\mathcal H(\mathbb B^n)$  the space of holomorphic functions on $\mathbb B^n.$

For $\alpha>-1$, we write $d\nu_{\alpha}(z)=c_{\alpha}(1-|z|^2)^{\alpha}d\nu(z)$, where $c_\alpha$ is taken such that $\nu_\alpha(\mathbb B^n)=1$. 

For $\Phi$ a growth function, the Orlicz space $L_\alpha^{\Phi}(\mathbb B^n)$ is the space of functions $f$ such that
$$||f||_{\alpha,\Phi}:=\int_{\mathbb B^n}\Phi(|f(z)|)d\nu_{\alpha}(z)<\infty.$$

The weighted Bergman-Orlicz space $\mathcal A_\alpha^{\Phi}(\mathbb B^n)$ is the subspace
of $L_\alpha^{\Phi}(\mathbb B^n)$ consisting of holomorphic functions.

We define on $\mathcal A_\alpha^{\Phi}(\mathbb B^n)$ the following (quasi)-norm
\begin{equation}\label{BergOrdef1}
||f||^{lux}_{\alpha,\Phi}:=\inf\{\lambda>0: \int_{\mathbb B^n}\Phi\left(\frac{|f(z)|}{\lambda}\right)d\nu_{\alpha}(z)\le 1\}
\end{equation}
which is finite for $f\in \mathcal A_\alpha^{\Phi}(\mathbb B^n)$ (see \cite{sehbastevic}).

We observe that for  $\Phi(t)=t^p$, the corresponding Bergman-Orlicz space is the classical weighted Bergman spaces denoted by $\mathcal A_\alpha^{p}(\mathbb B^n)$   and defined by
$$\|f\|_{p,\alpha}^p=||f||_{\mathcal A_\alpha^{p}}^p:= \int_{\mathbb B^n}|f(z)|^pd\nu_{\alpha}(z)<\infty.$$

Recall that two growth functions $\Phi_1$ and $\Phi_2$ are said equivalent if there exists some constant $c$ such
that
$$\frac{1}{c}\Phi_1(\frac{t}{c}) \le \Phi_2(t)\le c\Phi_1(ct)$$
and observe that two equivalent growth functions define the same Orlicz space.
\vskip .2cm


We recall
 that given an analytic function $f$ on $\mathbb B^n$, the
 radial derivative $\mathcal{R}f$ of $f$ is defined by
 $$\mathcal{R}f(z)=\sum_{j=1}^{n}z_j \frac{\partial f}{\partial z_j}(z).$$
 We recall also that the gradient of $f\in H(\mathbb{B}^n)$ is defined by $$\nabla f(z)=\left(\frac{\partial f}{\partial z_1}(z),\ldots, \frac{\partial f}{\partial z_n}(z)\right).$$
 The invariant gradient at $z$ of the analytic function $f$ is defined by $$\widetilde{\nabla}f(z)=\nabla(f\circ \phi_z)(0)$$
 where $\phi_z$ is the automorphism of $\mathbb{B}^n$ mapping $0$ to $z$.
 
 We have the following inequalities between the above derivatives (see \cite[Lemma 2.14]{KZ}):
 \Be\label{eq:derivativesineq}
 (1-|z|^2)|\mathcal{R}f(z)|\le (1-|z|^2)|\nabla f(z)|\le |\widetilde{\nabla} f(z)|,\,\,\,\textrm{for all}\,\,\, z\in \mathbb{B}^n.
 \Ee
The following derivatives characterization of classical weighted Bergman spaces is a well known fact (see \cite[Theorem 2.16]{KZ}).
\btheo
Suppose $\alpha>-1$, $p>0$, and $f$ is holomorphic in $\mathbb{B}^n$. Then the following conditions are equivalent.
\begin{itemize}
\item[(a)] $f\in \mathcal{A}_\alpha^p(\mathbb{B}^n)$.
\item[(b)] $|\widetilde{\nabla} f(z)|\in L^p(\mathbb{B}^n, d\nu_\alpha)$
\item[(c)] $(1-|z|^2)|\nabla f(z)|\in L^p(\mathbb{B}^n, d\nu_\alpha)$.
\item[(d)] $(1-|z|^2)|\mathcal{R}f(z)|\in L^p(\mathbb{B}^n, d\nu_\alpha).$
\end{itemize}
\etheo
\vskip .2cm
Our main aim in this note is to extend the above result to Bergman-Orlicz spaces. Let us recall some more definitions.

We say that a growth function $\Phi$ is of upper type  $q \geq 1$ if there exists $C>0$ such that, for $s>0$ and $t\ge 1$,
\begin{equation}\label{uppertype}
 \Phi(st)\le Ct^q\Phi(s).\end{equation}
We denote by $\mathscr{U}^q$ the set of growth functions $\Phi$ of upper type $q$, (for some $q\ge 1$), such that the function $t\mapsto \frac{\Phi(t)}{t}$ is non-decreasing.

We say that $\Phi$ is of lower type $p > 0$ if there exists $C>0$ such that, for $s>0$ and $0<t\le 1$,
\begin{equation}\label{eq:lowertype}
 \Phi(st)\le Ct^p\Phi(s).\end{equation}
We denote by $\mathscr{L}_p$ the set of growth functions $\Phi$ of lower type $p$,  (for some $p\le 1$), such that the function $t\mapsto \frac{\Phi(t)}{t}$ is non-increasing.

We also observe  that we may always suppose that any $\Phi\in \mathscr{L}_p$ (resp. $\mathscr{U}_q$),  is concave (resp. convex) and
that $\Phi$ is a $\mathscr{C}^1$ function with derivative $\Phi^{\prime}(t)\backsimeq \frac{\Phi(t)}{t}$. 
\vskip .2cm


Our main result is the following.
\btheo\label{thm:main}
Suppose $\alpha>-1$. Assume that $\Phi\in \mathscr{U}^q\cup \mathscr{L}_p$, and $f$ is holomorphic in $\mathbb{B}^n$. Then the following conditions are equivalent.
\begin{itemize}
\item[(a)] $f\in \mathcal{A}_\alpha^\Phi(\mathbb{B}^n)$.
\item[(b)] $|\widetilde{\nabla} f(z)|\in L^\Phi(\mathbb{B}^n, d\nu_\alpha)$
\item[(c)] $(1-|z|^2)|\nabla f(z)|\in L^\Phi(\mathbb{B}^n, d\nu_\alpha)$.
\item[(d)] $(1-|z|^2)|\mathcal{R}f(z)|\in L^\Phi(\mathbb{B}^n, d\nu_\alpha).$
\end{itemize}
\etheo
As applications, we characterize the Gustavsson-Peetre interpolate of two Bergman-Orlicz spaces and symbols of bounded Ces\`aro-type operators on Bergman-Orlicz spaces.
\section{Preliminary results}
We give in this section some useful tools needed in our presentation.
\subsection{Some properties of growth functions}
We recall that the complementary function $\Psi$ of the convex growth function $\Phi$, is the function defined from $\mathbb R_+$ onto itself by
\begin{equation}\label{complementarydefinition}
\Psi(s)=\sup_{t\in\mathbb R_+}\{ts - \Phi(t)\}.
\end{equation}
We observe that if $\Phi\in \mathscr{U}^q$, then $\Psi$ is a growth function of lower type such that the function which $t\mapsto \frac{\Psi(t)}{t}$ is non-decreasing. 

We say that $\Phi$ satisfies the $\Delta_2$-condition if there exists a constant $K>1$ such that, for any $t\ge 0$,
\begin{equation}\label{eq:delta2condition}
 \Phi(2t)\le K\Phi(t).\end{equation}

We say that the growth function $\Phi$ satisfies the $\bigtriangledown_2-$condition whenever both $\Phi$ and its complementary satisfy the $\Delta_2-$conditon.
\vskip .2cm
For $\Phi$ a $\mathcal C^1$ growth function, the lower and the upper indices of $\Phi$ are respectively defined by
$$a_\Phi:=\inf_{t>0}\frac{t\Phi^\prime(t)}{\Phi(t)}\,\,\,\textrm{and}\,\,\,b_\Phi:=\sup_{t>0}\frac{t\Phi^\prime(t)}{\Phi(t)}.$$
We recall that when $\Phi$ is convex, then $1\le a_\Phi\le b_\Phi<\infty$ and, if $\Phi$ is concave, then $0<a_\Phi\le b_\Phi\le 1$. We observe with \cite[Lemma 2.6]{DHZZ} that a convex growth function satisfies the $\bigtriangledown_2-$condition if and only if $1< a_\Phi\le b_\Phi<\infty$. 
\vskip .2cm
It is easy to see that if $\Phi$ is a $\mathcal C^1$ growth function.  Then the functions $\frac{\Phi(t)}{t^{a_\Phi}}$ and $\frac{\Phi^{-1}(t)}{t^{\frac{1}{b_\Phi}}}$ are increasing. As a consequence, we have the following useful fact.

\blem\label{lem:phip}
Let $\Phi\in \mathscr{L}_p$. Then the growth function $\Phi_p$, defined by
$\Phi_p(t)=\Phi(t^{1/p})$ is in $\mathscr{U}^q$ for some $q\geq 1$. 
\elem
We also make the following observation.
\begin{proposition}\label{phiandinverse}
The following assertion holds:
\begin{center}
    $\Phi\in \mathscr{L}_p$ if and only if $\Phi^{-1}\in \mathscr{U}^{1/p}.$
\end{center}
\end{proposition}

We observe that if $\Phi$ is of upper type (resp. lower type) $p_1$, then it is of upper type (resp. lower type) $p_2$ for any $\infty >p_2>p_1$ (resp. $p_2<p_1<\infty$). Hence, when we say $\Phi\in \mathscr{U}^q$ (resp. $\Phi\in \mathscr{L}_p$), we suppose that $q$ (resp. $p$) is the smallest (resp. biggest) number $q_1$ (resp. $p_1$) such that $\Phi$ is of upper type $q_1$ (resp. lower type $p_1$). We also observe that $a_\Phi$ (resp. $b_\Phi$) coincides with the biggest (resp. smallest) number $p$ such that $\Phi$ is of lower (resp. upper) type $p$.
\subsection{Operators on Orlicz spaces}
\begin{definition}
Let $\Phi$ be a growth function. A
linear operator $T$ defined on $L^\Phi(\mathbb{B}^n, d\nu_\alpha)$ is said to be of mean strong type $(\Phi, \Phi)_\alpha$ if
\Be\label{eq:meanstrong}
\int_{\mathbb{B}^n}\Phi(|Tf|) d\nu_\alpha(z)\le C\int_{\mathbb{B}^n}\Phi(|f|) d\nu_\alpha(z)
\Ee
for any $f\in L^\Phi(\mathbb{B}^n, d\nu_\alpha)$, and $T$ is said to be mean weak type $(\Phi,\Phi)_\alpha$ if
\Be\label{eq:meanweak} \sup_{t>0}\Phi(t)\nu_\alpha\left(\{z\in \mathbb{B}^n : |Tf(z)| > t\}\right)\le C\int_{\mathbb{B}^n}\Phi(|f|) d\nu_\alpha(z)
\Ee
for any $f\in L^\Phi(\mathbb{B}^n, d\nu_\alpha)$, where C is independent of f.
\end{definition}
We observe that the mean strong type $(t^p, t^p)_\alpha$ is the usual strong type $(p, p)$ coincide. We also note  if the operator $T$ is of mean strong type $(\Phi, \Phi)_\alpha$, then $T$ is bounded on $L^\Phi(\mathbb{B}^n,d\nu_\alpha)$.
\vskip .2cm
The following result is adapted from \cite[Theorem 4.3]{DHZZ}.
\btheo\label{thm:interpolation}
Let $\Phi_0, \Phi_1$ and $\Phi_2$ be three convex growth functions. Suppose that their
upper and lower indices satisfy the following condition
\Be\label{eq:lowerupperindicesrel} 1\le a_{\Phi_0}\le b_{\Phi_0}<a_{\Phi_2}\le b_{\Phi_2}<a_{\Phi_1}\le b_{\Phi_1}<\infty.
\Ee
If T is of mean weak types $(\Phi_0,\Phi_0)_\alpha$ and $(\Phi_1,\Phi_1)_\alpha$, then it is of mean strong type
$(\Phi_2, \Phi_2)_\alpha$.
\etheo
Let $\beta>-1$ be and consider the operator $P_\beta$ defined for functions $f$ on $\mathbb{B}^n$ by
\begin{equation} P_\beta(f)(z)=\int_{\mathbb B^n}\frac{f(\xi)}{(1-\langle z,\xi \rangle)^{n+1+\beta}}d\nu_\beta(\xi).\end{equation}
The operator $P_\beta$ is the Bergman projection, that is the orthogonal projection of $L^2(\mathbb{B}^n,d\nu_\beta)$ onto its closed subspace $\mathcal{A}_\beta^2(\mathbb{B}^n)$. We have the following result.
\btheo\label{thm:bdedfamillyPhi}
Let $\alpha, \beta>-1$. Let $\Phi$ be a convex growth function and denote by $a_\Phi$ its lower indice. Assume that there is $1<p_0<a_\Phi$ such that
$\alpha+1<p_0(\beta+1)$. Then $P_\beta$ is of mean strong type $(\Phi,\Phi)_\alpha$.
\etheo
\begin{proof}
This result is well known when $\Phi$ is a power function (see for example \cite[Theorem 2.10]{KZ}). It follows in particular that $P_\beta$ is bounded on $L^{p_0}(\mathbb{B}^n, d\nu_\alpha)$ and on $L^{p_1}(\mathbb{B}^n, d\nu_\alpha)$ for $p_1>b_\Phi$. Hence from the interpolation result Theorem \ref{thm:interpolation} we deduce that $P_\beta$ is of mean strong type $(\Phi,\Phi)_\alpha$.
\end{proof}
In particular, we have the following.
\btheo\label{thm:bdednessbergproj}
Let $\alpha>-1$. Assume that $\Phi\in \mathscr{U}^q$ and satisfies the $\bigtriangledown_2-$condition. Then the  Bergman projection $P_\alpha$ is bounded on $L^{\Phi}(\mathbb B^n, d\nu_\alpha)$.
\etheo
\subsection{Some useful estimates and test functions}
The next proposition gives pointwise estimates for functions in
$\mathcal A_{\alpha}^{\Phi}(\mathbb B^n)$, $\Phi\in \mathscr{L}_p\cup \mathscr{U}^q$ (see \cite{sehbastevic, ST2})
\blem\label{lem:pointwiseestimate}
Let $\Phi\in \mathscr{L}_p\cup \mathscr{U}^q$ and $\alpha>-1$. There is a constant $C>1$ such that
 for any $f\in \mathcal A_{\alpha}^{\Phi}(\mathbb B^n)$,
\begin{equation}\label{eq:pointwiseestimate}
|f(z)|\le C\Phi^{-1}\left(\frac 1{(1-|z|^2)^{n+1+\alpha}}\right)\|f\|_{\alpha,\Phi}^{lux}.
\end{equation}
\elem
We refer to \cite[Lemma 2.15]{ZZ} for the following result.
\blem\label{lem:smallestim}
 Let $0<p\le 1$. Then there is a constant $C>0$ such that for any $f\in \mathcal{A}_\alpha^{p}(\mathbb{B}^n)$, 
\begin{equation}\label{eq:intestimsmall}
\int_{\mathbb B^n}|f(z)|(1-|z|^2)^{(\frac{1}{p}-1)(n+1+\alpha)}d\nu_\alpha(z)\le C\|f\|_{\alpha,p}^p.
\end{equation}
\elem

The following gives example of functions in Bergman-Orlicz spaces. We refer to \cite{sehbastevic,ST2} for a proof.
 \begin{lemma}\label{lem:testfunctionberg}
Let $-1<\alpha<\infty$, $a\in \mathbb B^n$. Let $k>1$. Suppose that $\Phi\in \mathscr {L}_p\cup \mathscr {U}^q$. Then the following function is in $\mathcal A_{\alpha}^{\Phi}(\mathbb B^n)$
$$f_a(z)=\Phi^{-1}\left(\frac{1}{(1-|a|)^{n+1+\alpha}}\right)\left(\frac{1-|a|^2}{1-\langle z,a\rangle}\right)^{k(n+1+\alpha)}.$$
Moreover, $||f_a||_{\mathcal A_{\alpha}^{\Phi}}^{lux}\lesssim 1.$
\end{lemma}

\section{Proof of Theorem \ref{thm:main}}
Let us start with the following result.
\blem\label{lem:invargradestim}
Let $\alpha>-1$. Assume that $\Phi\in \mathscr{U}^q\cup \mathscr{L}_p$. Then there exists a constant $C>0$ such that for any $f\in \mathcal{A}_\alpha^\Phi(\mathbb{B}^n)$,
\Be\label{eq:invargradestim}
\int_{\mathbb{B}^n}\Phi(|\widetilde{\nabla}f(z)|)d\nu_\alpha(z)\le C\int_{\mathbb{B}^n}\Phi(|f(z)-f(0)|)d\nu_\alpha(z).
\Ee
\elem
\begin{proof}
We follow the proof of \cite[Theorem 2.16]{KZ} making some crucial modifications where needed. We start by recalling that if $\Phi\in \mathscr{U}^q$, then $\mathcal{A}_\alpha^\Phi(\mathbb{B}^n)$ continuously embeds into $\mathcal{A}_\alpha^1(\mathbb{B}^n)$, and $\mathcal{A}_\alpha^\Phi(\mathbb{B}^n)$ continuously embeds into $\mathcal{A}_\alpha^p(\mathbb{B}^n)$ when $\Phi\in \mathscr{L}_p$ . Let $\beta>\alpha$. Put \Be\label{eq:pPhi} p_\Phi= \left\{\begin{array}{lcr}1 & \mbox{if} & \Phi\in \mathscr{U}^q\\ p & \mbox{if} & \Phi\in \mathscr{L}_p.\end{array}\right.\Ee It follows from \cite[Lemma 2.4]{KZ} that there exists $C_1>0$ such that for any $g\in H(\mathbb{B}^n)$,
$$|\nabla g(0)|^{p_\Phi}\le C_1\int_{\mathbb{B}^n}|g(w)|^{p_\Phi}d\nu_\beta(w).$$
Put $g=f\circ \phi_z$, $z\in \mathbb{B}^n$, where $\phi_z$ is the automorphism of $\mathbb{B}^n$ such that $\phi_z(0)=z$. We obtain
$$|\widetilde{\nabla} f(z)|^{p_\Phi}\le C_1\int_{\mathbb{B}^n}|f(w)|^{p_\Phi}\frac{(1-|z|^2)^{n+1+\beta}}{|1-\langle z,w\rangle|^{2(n+1+\beta)}} d\nu_\beta(w).$$
We observe with the help of \cite[Proposition 1.4.10]{R} that $\frac{(1-|z|^2)^{n+1+\beta}}{|1-\langle z,w\rangle|^{2(n+1+\beta)}} d\nu_\beta(w)$ is up to a constant a probability measure. It follows using the convexity of $$\Phi_p(t)= \left\{\begin{array}{lcr}\Phi(t) & \mbox{if} & \Phi\in \mathscr{U}^q\\ \Phi(t^\frac{1}{p}) & \mbox{if} & \Phi\in \mathscr{L}_p\end{array}\right.$$ and Jensen's Inequality that
$$\Phi(|\widetilde{\nabla} f(z)|)\le C_2\int_{\mathbb{B}^n}\Phi(|f(w)|)\frac{(1-|z|^2)^{n+1+\beta}}{|1-\langle z,w\rangle|^{2(n+1+\beta)}} d\nu_\beta(w).$$
Finally, integrating both sides of the last inequality over $\mathbb{B}^n$ with respect to $d\nu_\alpha(z)$ and applying Fubini's Theorem and \cite[Proposition 1.4.10]{R}, we obtain
$$\int_{\mathbb{B}^n}\Phi(|\widetilde{\nabla} f(z)|)d\nu_\alpha \le C_2\int_{\mathbb{B}^n}\Phi(|f(z)|)d\nu_\alpha(z).$$

Replacing $f$ by $f-f(0)$, we have 
$$\int_{\mathbb{B}^n}\Phi(|\widetilde{\nabla} f(z)|)d\nu_\alpha \le C_2\int_{\mathbb{B}^n}\Phi(|f(z)-f(0)|)d\nu_\alpha(z).$$
The proof is complete.
\end{proof}
We also obtain the following.
\blem\label{lem:normdominatedbyradialderiv}
Let $\alpha>-1$. Assume that $\Phi\in \mathscr{U}^q$ or $\Phi\in \mathscr{L}_p$. Then there exists a constant $C>0$ such that for any $f\in H(\mathbb{B}^n)$ such that $(1-|z|^2)\mathcal{R}f(z)\in L^\Phi(\mathbb{B}^n, d\nu_\alpha)$,
\Be\label{eq:normdominatedbyradialderiv}
\int_{\mathbb{B}^n}\Phi(|f(z)-f(0)|)d\nu_\alpha(z)\le C\int_{\mathbb{B}^n}\Phi((1-|z|^2)|\mathcal{R}f(z)|)d\nu_\alpha(z).
\Ee
\elem
\begin{proof}
Let $f\in H(\mathbb{B}^n)$ be such that $(1-|z|^2)\mathcal{R}f(z)\in L^\Phi(\mathbb{B}^n, d\nu_\alpha)$. Then following  the proof of \cite[Theorem 2.16]{KZ} at page 51, we have that for $\beta$ large enough,
\Be\label{eq:page51}
|f(z)-f(0)|\le \int_{\mathbb{B}^n}\frac{|\mathcal{R}f(w)|}{|1-\langle z,w\rangle|^{n+\beta}}d\nu_\beta(w).
\Ee
Let us first consider the case of $\Phi\in \mathscr{U}^q$. Fix $p$ so that $1<p<a_\Phi$, and observe that (\ref{eq:page51}) is equivalent to
$$|f(z)-f(0)|\le CP_{\beta-1}((1-|\cdot|^2)|\mathcal{R}f(\cdot)|)(z).$$
Taking $\beta$ large enough so that $$0<\alpha+1<p\beta,$$
we obtain from Theorem \ref{thm:bdedfamillyPhi} that
$$\int_{\mathbb{B}^n}\Phi(|f(z)-f(0)|)d\nu_\alpha(z)\le C\int_{\mathbb{B}^n}\Phi((1-|z|^2)|\mathcal{R}f(z)|)d\nu_\alpha(z).$$
We next consider the case of $\Phi\in \mathscr{L}_p$. We assume that $\beta$ is large enough so that
$$\beta=\frac{n+1+\gamma}{p}-(n+1),\,\,\,\gamma>\alpha+p.$$
Rewriting (\ref{eq:page51}) as 
$$|f(z)-f(0)|\le \int_{\mathbb{B}^n}\left|\frac{\mathcal{R}f(w)}{(1-\langle z,w\rangle)^{n+\beta}}\right|(1-|w|^2)^{(\frac{1}{p}-1)(n+1+\gamma)}d\nu(w),$$
we obtain from Lemma \ref{lem:smallestim} that
$$|f(z)-f(0)|^p\le C\int_{\mathbb{B}^n}\left|\frac{\mathcal{R}f(w)}{(1-\langle z,w\rangle)^{n+\beta}}\right|^pd\nu_\gamma(w),$$
or equivalently,
\Beas |f(z)-f(0)|^p &\le& C\int_{\mathbb{B}^n}\frac{\left((1-|w|^2)|\mathcal{R}f(w)|\right)^p}{|1-\langle z,w\rangle|^{n+1+(\gamma-p)}}d\nu_{\gamma-p}(w)\\ &=& CP_{\gamma-p}\left([(1-|\cdot|^2)|\mathcal{R}f(\cdot)|]^p\right)(z).
\Eeas
As the growth function $t\mapsto \Phi_p(t)=\Phi(t^\frac{1}{p})$ is in $\mathscr{U}^q$, proceeding as in the first part of this proof, we obtain 
\Beas
\int_{\mathbb{B}^n}\Phi(|f(z)-f(0)|)d\nu_\alpha(z) &\le& C\int_{\mathbb{B}^n}\Phi_p\left([(1-|\cdot|^2)|\mathcal{R}f(\cdot)|]^p\right)(z)d\nu_\alpha(z)\\ &=& C\int_{\mathbb{B}^n}\Phi((1-|z|^2)|\mathcal{R}f(z)|)d\nu_\alpha(w).
\Eeas
The proof is complete.

\end{proof}
We can now prove Theorem \ref{thm:main}.
\begin{proof}[Proof of Theorem \ref{thm:main}]
That (b) implies (c) and (c) implies (d) follow from (\ref{eq:derivativesineq}). That (a) implies (b) is Lemma \ref{lem:invargradestim} and that (d) implies (a) is Lemma \ref{lem:normdominatedbyradialderiv}. The proof is complete.

\end{proof}
\section{Applications}
\subsection{The Gustavsson-Peetre interpolation of two Bergman-Orlicz spaces}
Our aim in this section is to give an application of Theorem \ref{thm:main} to a generalized interpolation 
of quasi-Banach spaces due to J. Gustavsson and J. Peetre. For this, we first introduce several definitions and results.
\begin{definition}
A function $\rho: [0,\infty)\rightarrow [0,\infty)$ is said to be pseudo-concave, if it is continuous on $(0,\infty)$ and $$\rho(s)\le \max(1,\frac{s}{t})\rho(t),\,\,\,\textrm{for all}\,\,\,s,t>0.$$
 We denote by $\mathcal{T}$ the set of all pseudo-concave functions.
\end{definition} 

\begin{definition}
A function $\rho\in \mathcal{T}$ is said to be in $\mathcal{T}^{+-}$, if 
 $$\sup_{x}\frac{\rho(\lambda x)}{\rho(x)}=o(\max(1,\lambda))\,\,\,\textrm{as}\,\,\,x\rightarrow 0\,\,\,\textrm{or}\,\,\,x\rightarrow \infty.$$
\end{definition} 

As example of functions in $\mathcal{T}^{+-}$ we have of course concave functions. The following function was provided in \cite{Gustavsson} as a non trivial element in $\mathcal{T}^{+-}$: $$\rho(t)=t^\theta(\log(e+t))^\alpha(e+\frac{1}{t})^\beta$$ where $0<\theta<1$, and $\alpha, \beta$ are real numbers.

\begin{definition}[J. Gustavsson and J. Peetre \cite{Gustavsson}]
Let $A_0$ and $A_1$ be quasi-Banach spaces, both embedded in a Hausdorff topological space $\mathcal{A}$. We call $\vec{A}=(A_0,A_1)$ a quasi-Banach couple. Let $\rho\in \mathcal{T}$. We denote by $\langle\vec{A}\rangle_\rho=\langle A_0,A_1\rangle_\rho$ the space of all elements $a\in \sum(\vec{A}):=A_0+A_1$ such that there exists a sequence $u=\{u_\alpha\}_{\alpha\in \mathbb{Z}}$ of elements $u_\alpha\in \Delta(\vec{A}):=A_0\cap A_1$ such that
\Be
a=\sum_{\alpha\in \mathbb{Z}}u_\alpha\,\,\,(\textrm{convergence in}\,\,\,\sum(\vec{A})),
\Ee
for every finite subset $F\subset \mathbb{Z}$ and every real sequence $\xi=\{\xi_\alpha\}_{\alpha\in F}$ with $|\xi_\alpha|\le 1$, we have
$$\|\sum_{F}\frac{\xi_\alpha 2^{k\alpha}u_\alpha}{\rho(2^\alpha)}\|\le C\,\,\,(k=0,1)$$
with $C$ independent of $F$ and $\xi$.

We equip the space $\langle\vec{A}\rangle_\rho$ with the semi-norm $$\|a\|_{\langle\vec{A}\rangle_\rho}:= \inf_{u} C,$$
where the  infimum is taken over all admissible sequences $u$ as above.  
\end{definition} 
As observed in \cite{Gustavsson}, if $\rho\in \mathcal{T}^{+-}$, then $\|a\|_{\langle\vec{A}\rangle_\rho}$ is a quasi-norm and $\langle\vec{A}\rangle_\rho$ is a quasi-Banach space.

Let us recall that if $T: A\rightarrow B$ is a continuous linear operator between two quasi-Banach spaces $A$ and $B$, the operator norm of $T$ denoted $\|T\|_{A\rightarrow B}$ is defined by $$\|T\|_{A\rightarrow B}:=\sup_{x\in A, x\neq 0}\frac{\|Tx\|_{B}}{\|x\|_A}.$$
\begin{proposition}\textsc{(\cite[Proposition 6.1.]{Gustavsson})}\label{prop:functorinterp}
Let $\vec{A}=(A_0,A_1)$ and $\vec{B}=(B_0,B_1)$ be two quasi-Banach couples. If $T:\vec{A}\rightarrow \vec{B}$ is a continuous linear mapping, that is the restriction $T_i=T|A_i: A_i\rightarrow B_i$, ($i=0,1$) is continuous, then $T:\langle\vec{A}\rangle_\rho\rightarrow \langle\vec{B}\rangle_\rho$ is continuous and $$\|T\|_{\langle\vec{A}\rangle_\rho\rightarrow \langle\vec{B}\rangle_\rho}\le \max(\|T\|_{A_0\rightarrow B_0}, \|T\|_{A_1\rightarrow B_1}).$$
That is the functor $(A_0,A_1)\mapsto \langle\vec{A}\rangle_\rho$ is an interpolation space.
\end{proposition}
We call $\langle\vec{A}\rangle_\rho$ the Gustavsson-Peetre interpolate of $A_0$ and $A_1$. As observed in \cite{Gustavsson}, when $\rho(s)=s^{\theta}$, $0< \theta<1$, $\langle\vec{A}\rangle_\rho$ corresponds to the complex interpolation (see also \cite{Cobos,Kraynek,Rao}). For more on complex interpolation, we refer to the book \cite{Berg}. 

The following is a restriction of \cite[Theorem 7.3]{Gustavsson} to the class of growth functions we are interested in and our spaces.
\begin{proposition}\label{prop:GPinterorlicz}
Let $\Phi_i\in \mathscr{U}^q$, $i=0,1$ satisfying the $\triangledown_2-$condition, and let $\alpha>-1$. Assume that $\rho\in \mathcal{T}^{+-}$ and let $\Phi$ be defined by  
\Be\label{eq:phiinterp}\Phi^{-1}=\Phi_0^{-1}\rho(\frac{\Phi_1^{-1}}{\Phi_0^{-1}}).\Ee
 Then $$L^{\Phi}(\mathbb{B}^n,d\nu_\alpha)=\vec{L^{\Phi}}_\rho(\mathbb{B}^n,d\nu_\alpha)=\langle L^{\Phi_0}(\mathbb{B}^n,d\nu_\alpha),L^{\Phi_1}(\mathbb{B}^n,d\nu_\alpha)\rangle_\rho$$
with equivalence of (quasi)-norms. In particular, $\vec{L^{\Phi}}_\rho(\mathbb{B}^n,d\nu_\alpha)$ is an interpolation space with respect to $ \left(L^{\Phi_0}(\mathbb{B}^n,d\nu_\alpha),L^{\Phi_1}(\mathbb{B}^n,d\nu_\alpha)\right)$.
\end{proposition}  
It is easy to check using the definition of a pseudo-concave function, that given $\Phi_0,\Phi_1\in \mathscr{U}^q$, the growth function $\Phi$ defined by (\ref{eq:phiinterp}) is of upper type $q>1$.

We can now state our main result of this section.
\begin{theorem}\label{thm:GPBergOrlicz}
Let $\Phi_i\in \mathscr{U}^q$, $i=0,1$ satisfying the $\triangledown_2-$condition, and let $\alpha>-1$. Assume that $\rho\in \mathcal{T}^{+-}$ and let $\Phi$ be defined as in (\ref{eq:phiinterp}). Then $$\mathcal{A}_\alpha^{\Phi}(\mathbb{B}^n)=\langle \mathcal{A}_\alpha^{\Phi_0}(\mathbb{B}^n),\mathcal{A}_\alpha^{\Phi_1}(\mathbb{B}^n)\rangle_\rho$$
with equivalence of (quasi)-norms.
\end{theorem}
\begin{proof}
As $\Phi_0$ and $\Phi_1$ satisfy the $\triangledown_2-$condition, we have from Theorem \ref{thm:bdednessbergproj} that the Bergman projection $P_\alpha$ maps $L^{\Phi_0}(\mathbb{B}^n, d\nu_\alpha)$ boundedly onto $\mathcal{A}_\alpha^{\Phi_0}(\mathbb{B}^n)$, and it maps $L^{\Phi_1}(\mathbb{B}^n,d\nu_\alpha)$ boundedly onto $\mathcal{A}_\alpha^{\Phi_1}(\mathbb{B}^n)$. It follows from Proposition \ref{prop:functorinterp} that $P_\alpha$ maps $L^{\Phi}(\mathbb{B}^n, d\nu_\alpha)=\langle L^{\Phi_0}(\mathbb{B}^n, d\nu_\alpha),L^{\Phi_1}(\mathbb{B}^n, d\nu_\alpha)\rangle_\rho$ boundedly into $\langle \mathcal{A}_\alpha^{\Phi_0}(\mathbb{B}^n),\mathcal{A}_\alpha^{\Phi_1}(\mathbb{B}^n)\rangle_\rho$. As $P_\alpha\left(L^{\Phi}(\mathbb{B}^n, d\nu_\alpha)\right)=\mathcal{A}_\alpha^{\Phi}(\mathbb{B}^n)$, we conclude that 

\Be\label{eq:incluinterpberg}\mathcal{A}_\alpha^{\Phi}(\mathbb{B}^n)\subset \langle \mathcal{A}_\alpha^{\Phi_0}(\mathbb{B}^n),\mathcal{A}_\alpha^{\Phi_1}(\mathbb{B}^n)\rangle_\rho.
\Ee
Conversely, if we denote by $L$ the operator defined on $H(\mathbb{B}^n)$ by $$L(f)(z):=(1-|z|^2)\mathcal{R}f(z),\,\,\,z\in \mathbb{B}^n,$$
then following Theorem \ref{thm:main} and specially Lemma \ref{lem:normdominatedbyradialderiv}, we have that $L$ maps $\mathcal{A}_\alpha^{\Phi_0}(\mathbb{B}^n)$ boundedly into $L^{\Phi_0}(\mathbb{B}^n, d\nu_\alpha)$,  and it maps  $\mathcal{A}_\alpha^{\Phi_1}(\mathbb{B}^n)$ into $L^{\Phi_1}(\mathbb{B}^n, d\nu_\alpha)$. It follows once more from Proposition \ref{prop:functorinterp} that $L$ maps $\langle \mathcal{A}_\alpha^{\Phi_0}(\mathbb{B}^n),\mathcal{A}_\alpha^{\Phi_1}(\mathbb{B}^n)\rangle_\rho$ into $L^{\Phi}(\mathbb{B}^n, d\nu_\alpha)$. That is if $f\in \langle \mathcal{A}_\alpha^{\Phi_0}(\mathbb{B}^n),\mathcal{A}_\alpha^{\Phi_1}(\mathbb{B}^n)\rangle_\rho$, then the function $z\mapsto (1-|z|^2)\mathcal{R}f(z)$ belongs to $L^{\Phi}(\mathbb{B}^n, d\nu_\alpha)$, which by Theorem \ref{thm:main} is equivalent to saying that $f\in \mathcal{A}_\alpha^{\Phi}(\mathbb{B}^n)$. We deduce that 
\Be\label{eq:incluinverinterpberg}\langle \mathcal{A}_\alpha^{\Phi_0}(\mathbb{B}^n),\mathcal{A}_\alpha^{\Phi_1}(\mathbb{B}^n)\rangle_\rho\subset \mathcal{A}_\alpha^{\Phi}(\mathbb{B}^n).
\Ee
From (\ref{eq:incluinterpberg}) and (\ref{eq:incluinverinterpberg}), we conclude that 
$$\langle \mathcal{A}_\alpha^{\Phi_0}(\mathbb{B}^n),\mathcal{A}_\alpha^{\Phi_1}(\mathbb{B}^n)\rangle_\rho= \mathcal{A}_\alpha^{\Phi}(\mathbb{B}^n).$$
The proof is complete.
\end{proof}
Restricting to power functions, we deduce the following.
\begin{corollary}\label{cor:GPBergOrlicz}
Let $1\le p_0<p_1<\infty$, and let $\alpha>-1$. Assume that $\rho\in \mathcal{T}^{+-}$ and let $\Phi$ be defined by $\Phi^{-1}(t)=t^{\frac{1}{p_0}}\rho(t^{\frac{1}{p_1}-\frac{1}{p_0}})$. Then $$\mathcal{A}_\alpha^{\Phi}(\mathbb{B}^n)=\langle \mathcal{A}_\alpha^{p_0}(\mathbb{B}^n),\mathcal{A}_\alpha^{p_1}(\mathbb{B}^n)\rangle_\rho$$
with equivalence of (quasi)-norms.
\end{corollary}
Note that the above corollary tells us that given two classical Bergman spaces with the same weight, their Gustavsson-Peetre interpolation space is in general a Bergman-Orlicz space while the complex interpolation of weighted Bergman spaces always gives another weighted Bergman space (see \cite{KZ}). 

\subsection{Boundedness and compactness of weighted Ces\`aro-type integrals}
\vskip .2cm
For $g\in \mathcal H(\mathbb B^n)$ with $g(0)=0$, we consider the following integral-type operator defined on $\mathcal H(\mathbb B^n)$ by

$$T_g f(z)=\int_0^1f(tz)\mathcal{R}g(tz)\frac{dt}{t}.$$
The  operator $T_g$ is the so-called extended Ces\`aro operator introduced in \cite{Hu1}. The boundedness and compactness of $T_g$ on the weighted Bergman space $\mathcal{A}_\alpha^p(\mathbb{B}^n)$ were studied by J. Xiao \cite{Xiao}. In \cite{LiStevic} the same questions between different weighted Bergman spaces were studied. Note that Z. Hu also considered the boundedness and compactness of $T_g$ between weighted Bergman  spaces for a large class of weights \cite{Hu3}. 
\vskip .2cm
We aim in this section to characterize symbols $g$ such that $T_g$ is a bounded or compact operator from a weighted Bergman-Orlicz space to itself.
\vskip .2cm

We first prove an estimate for derivative of functions in Bergman-Orlicz spaces.
\blem\label{lem:pointwiseforderivatives}
Let $\Phi\in \mathscr{L}_p\cup \mathscr{U}^q$ and $\alpha>-1$. Then there are two positive constants $C_1$ and $C_2$ such that for any  $f\in \mathcal A_{\alpha}^{\Phi}(\mathbb B^n)$,
\Be\label{eq:pointwiseforderivatives}
|\nabla f(z)|\le \frac{C_1}{1-|z|^2}\Phi^{-1}\left(\frac{C_2}{(1-|z|^2)^{n+1+\alpha}}\right)\|f\|_{\alpha,\Phi}^{lux}, \,\,\,\textrm{for any}\,\,\,z\in \mathbb{B}^n.
\Ee
\elem
\begin{proof}
Let us start by considering the case where $\Phi\in \mathscr{U}^q$. We observe that in this case, $\mathcal A_{\alpha}^{\Phi}(\mathbb B^n)$ continuously embeds into $\mathcal A_{\alpha}^{1}(\mathbb B^n)$. Hence, for any $f\in \mathcal A_{\alpha}^{\Phi}(\mathbb B^n)$, and any $z\in \mathbb{B}^n$,
$$f(z)=\int_{\mathbb{B}^n}\frac{f(w)}{(1-\langle z,w\rangle)^{n+1+\alpha}}d\nu_\alpha(w).$$
Thus for any $j=1,\ldots,n$,
$$\frac{\partial f}{\partial z_j}(z)=c\int_{\mathbb{B}^n}\frac{\overline{w}_jf(w)}{(1-\langle z,w\rangle)^{n+2+\alpha}}d\nu_\alpha(w).$$
It follows easily that
$$\frac{1-|z|^2}{\|f\|_{\alpha,\Phi}^{lux}}|\frac{\partial f}{\partial z_j}(z)|\le C\int_{\mathbb{B}^n}\frac{|f(w)|}{\|f\|_{\alpha,\Phi}^{lux}}\frac{1-|z|^2}{(1-\langle z,w\rangle)^{n+2+\alpha}}d\nu_\alpha(w).$$
It is easy to see using \cite[Proposition 1.4.10]{R} that $\frac{1-|z|^2}{(1-\langle z,w\rangle)^{n+2+\alpha}}d\nu_\alpha(w)$ is up to a constant a probability measure. Hence using the convexity of $\Phi$ and  Jensen's Inequality, we obtain
\Beas \Phi\left(\frac{1-|z|^2}{\|f\|_{\alpha,\Phi}^{lux}}\left|\frac{\partial f}{\partial z_j}(z)\right|\right) &\le& C\int_{\mathbb{B}^n}\Phi\left(\frac{|f(w)|}{\|f\|_{\alpha,\Phi}^{lux}}\right)\frac{1-|z|^2}{(1-\langle z,w\rangle)^{n+2+\alpha}}d\nu_\alpha(w)\\ &\le& \frac{C}{(1-|z|^2)^{n+1+\alpha}}\int_{\mathbb{B}^n}\Phi\left(\frac{|f(w)|}{\|f\|_{\alpha,\Phi}^{lux}}\right)d\nu_\alpha(w)\\ &\le& \frac{C}{(1-|z|^2)^{n+1+\alpha}}.
\Eeas

Hence
$$\left|\frac{\partial f}{\partial z_j}(z)\right|\le \frac{1}{1-|z|^2}\Phi^{-1}\left(\frac C{(1-|z|^2)^{n+1+\alpha}}\right)\|f\|_{\alpha,\Phi}^{lux}, \,\,\,\textrm{for any}\,\,\,z\in \mathbb{B}^n$$
from which follows (\ref{eq:pointwiseforderivatives}).
\vskip .3cm
We now consider the case where $\Phi\in \mathscr{L}_p$. We recall that in this case $\Phi$ is of lower type $0<p\le 1$. Let $\beta>-1$ be large enough (this will be more precise in the next lines). As above, we have that
\Be\label{eq:modulederivestim}\left|\frac{\partial f}{\partial z_j}(z)\right|\le C\int_{\mathbb{B}^n}\frac{|f(w)|}{|1-\langle z,w\rangle|^{n+2+\beta}}d\nu_\beta(w).\Ee
We assume that $\beta=\frac{n+1+\gamma}{p}-(n+1)$ with $\gamma>\alpha+p$. Then using Lemma \ref{lem:smallestim}, we obtain from (\ref{eq:modulederivestim}) that
\Be \left|\frac{\partial f}{\partial z_j}(z)\right|^p\le C\int_{\mathbb{B}^n}\left|\frac{f(w)}{(1-\langle z,w\rangle)^{n+2+\beta}}\right|^pd\nu_\gamma(w)
\Ee
or equivalently,
\Be\label{eq:applismallexpo}\left|\frac{1-|z|^2}{\|f\|_{\alpha,\Phi}^{lux}}\frac{\partial f}{\partial z_j}(z)\right|^p\le C\int_{\mathbb{B}^n}\left|\frac{f(w)}{\|f\|_{\alpha,\Phi}^{lux}}\right|^p\frac{(1-|z|^2)^p}{|1-\langle z,w\rangle)|^{(n+2+\beta)p}}d\nu_\gamma(w).
\Ee
One easily checks that $\frac{(1-|z|^2)^p}{|1-\langle z,w\rangle)|^{(n+2+\beta)p}}d\nu_\gamma(w)$ is up to a constant, a probability measure. Hence using that the function $\Phi_p:t\mapsto \Phi_p(t)=\Phi(t^{\frac{1}{p}})$ is convex and Jensen's Inequality, we obtain that
$$\Phi_p\left(\left|\frac{1-|z|^2}{\|f\|_{\alpha,\Phi}^{lux}}\frac{\partial f}{\partial z_j}(z)\right|^p\right)\le C\int_{\mathbb{B}^n}\Phi_p\left(\left|\frac{f(w)}{\|f\|_{\alpha,\Phi}^{lux}}\right|^p\right)\frac{(1-|z|^2)^p}{|1-\langle z,w\rangle)|^{(n+2+\beta)p}}d\nu_\gamma(w).$$
Hence
\Beas\Phi\left(\left|\frac{1-|z|^2}{\|f\|_{\alpha,\Phi}^{lux}}\frac{\partial f}{\partial z_j}(z)\right|\right) &\le& C\int_{\mathbb{B}^n}\Phi\left(\left|\frac{f(w)}{\|f\|_{\alpha,\Phi}^{lux}}\right|\right)\frac{(1-|z|^2)^p}{|1-\langle z,w\rangle)|^{(n+2+\beta)p}}d\nu_\gamma(w)\\ &=& C\int_{\mathbb{B}^n}\Phi\left(\left|\frac{f(w)}{\|f\|_{\alpha,\Phi}^{lux}}\right|\right)\frac{(1-|z|^2)^p}{|1-\langle z,w\rangle)|^{n+1+\gamma+p}}d\nu_\gamma(w)\\ &\le& \frac{C}{(1-|z|^2)^{n+1+\alpha}}\int_{\mathbb{B}^n}\Phi\left(\left|\frac{f(w)}{\|f\|_{\alpha,\Phi}^{lux}}\right|\right)d\nu_\alpha(w)\\ &\le& \frac{C}{(1-|z|^2)^{n+1+\alpha}}.
\Eeas
That is $$\left|\frac{\partial f}{\partial z_j}(z)\right|\le \frac{1}{1-|z|^2}\Phi^{-1}\left(\frac C{(1-|z|^2)^{n+1+\alpha}}\right)\|f\|_{\alpha,\Phi}^{lux}, \,\,\,\textrm{for any}\,\,\,z\in \mathbb{B}^n$$
from which follows (\ref{eq:pointwiseforderivatives}).
The proof is complete. 
\end{proof}
We can now prove the following.
\btheo\label{thm:bdcesaro}
Let $\Phi\in \mathscr{L}_p\cup \mathscr{U}^q$ and $\alpha>-1$. Assume $g$ is a holomorphic function on $\mathbb{B}^n$ with $g(0)=0$. Then the operator $T_g$ is bounded on $\mathcal{A}_\alpha^\Phi(\mathbb{B}^n)$ if and only if
 
\Be\label{eq:blochcondi}
M:=\sup_{z\in \mathbb{B}^n}(1-|z|^2)|\mathcal{R}g(z)|<\infty.
\Ee

Moreover, if we denote by $\|T_g\|$ the operator norm of $T_g$, then $$\|T_g\|\backsim M.$$
\etheo
\begin{proof}
Let us first assume that (\ref{eq:blochcondi}) holds. Then
\Beas
I &:=& \int_{\mathbb{B}^n}\Phi\left(\frac{(1-|z|^2)|\mathcal{R}T_gf(z)|}{M\|f\|_{\Phi,\alpha}^{lux}}\right)d\nu_\alpha(z)\\ &=& \int_{\mathbb{B}^n}\Phi\left(\frac{(1-|z|^2)|\mathcal{R}g(z)|}{M}\frac{|f(z)|}{\|f\|_{\Phi,\alpha}^{lux}}\right)d\nu_\alpha(z)\\ &\le& \int_{\mathbb{B}^n}\Phi\left(\frac{|f(z)|}{\|f\|_{\Phi,\alpha}^{lux}}\right)d\nu_\alpha(z)\le 1.
\Eeas
Hence from Theorem \ref{thm:main}, we deduce that $T_gf\in \mathcal{A}_\alpha^\Phi(\mathbb{B}^n)$ for any $f\in \mathcal{A}_\alpha^\Phi(\mathbb{B}^n)$. Moreover, from Lemma \ref{lem:normdominatedbyradialderiv}, we deduce that 
$$\|T_gf\|_{\Phi,\alpha}^{lux}\le M\|f\|_{\Phi,\alpha}^{lux}$$
for any $f\in \mathcal{A}_\alpha^\Phi(\mathbb{B}^n)$. It follows that $$\|T_g\|\le M.$$
Conversely, let us assume that for $g\in H(\mathbb{B}^n)$ with $g(0)=0$, $T_g$ is bounded on $\mathcal{A}_\alpha^\Phi(\mathbb{B}^n)$. Then from Lemma \ref{lem:pointwiseforderivatives} we have that there is a constant $C>0$ such that for any $f\in \mathcal{A}_\alpha^\Phi(\mathbb{B}^n)$ and any $z\in \mathbb{B}^n$,
$$(1-|z|^2)|\mathcal{R}T_gf(z)|\le C\Phi^{-1}\left(\frac{1}{(1-|z|^2)^{n+1+\alpha}}\right)\|T_gf\|_{\Phi,\alpha}^{lux}$$
which leads to
\Be\label{eq:necessbd}(1-|z|^2)|\mathcal{R}g(z)||f(z)|\le C\Phi^{-1}\left(\frac{1}{(1-|z|^2)^{n+1+\alpha}}\right)\|T_g\|\|f\|_{\Phi,\alpha}^{lux}.
\Ee
Let $a\in \mathbb{B}^n$ be fixed and consider the function $f_a$ defined on $\mathbb{B}^n$ by
$$f_a(z)=\Phi^{-1}\left(\frac{1}{(1-|a|^2)^{n+1+\alpha}}\right)\left(\frac{1-|a|^2}{1-\langle z,a\rangle}\right)^{2(n+1+\alpha)}.$$
We recall with Lemma \ref{lem:testfunctionberg} that $f_a\in \mathcal{A}_\alpha^\Phi(\mathbb{B}^n)$ and $\|f_a\|_{\Phi,\alpha}^{lux}\lesssim 1$.
Let us test (\ref{eq:necessbd}) with the function $f=f_a$, for $a$ fixed. We obtain that for any $z\in \mathbb{B}^n$,
\Beas && (1-|z|^2)|\mathcal{R}g(z)|\Phi^{-1}\left(\frac{1}{(1-|a|^2)^{n+1+\alpha}}\right)\left|\frac{1-|a|^2}{1-\langle z,a\rangle}\right|^{2(n+1+\alpha)}\\ &\le& C.\Phi^{-1}\left(\frac{1}{(1-|z|^2)^{n+1+\alpha}}\right)\|T_g\|.
\Eeas
Putting $z=a$, we obtain that 
$$ (1-|a|^2)|\mathcal{R}g(a)|\le C\|T_g\|.$$
As $a$ was taken arbitrary in $\mathbb{B}^n$, we deduce that there is a constant $C>0$ such 
$$ M:=\sup_{a\in \mathbb{B}^n}(1-|a|^2)|\mathcal{R}g(a)|\le C\|T_g\|.$$
The proof is complete.
\end{proof}
Recall that that the spaces of all holomorphic functions satisfying (\ref{eq:blochcondi}) is called the Bloch space. The equivalent characterizations in Theorem \ref{thm:main} show that the Bloch space embeds continuously into $\mathcal A_\alpha^{\Phi}(\mathbb B^n)$ for any $\Phi\in \mathscr U^q\cup \mathscr L_p$.
\vskip .3cm
We next consider compactness of the operators $T_g$. For this, we need the following compactness criteria which can be proved following the usual arguments (see \cite{CM}).
\begin{lemma}\label{lem compactcriteria}
Let $\Phi\in \mathscr U^q\cup \mathscr L_p$, and let $\alpha>-1$.
Let  $g\in  H(\mathbb B^n)$ with $g(0)=0$. Suppose that  ${T_g:
 \mathcal A_\alpha^{\Phi}(\mathbb B^n)\rightarrow \mathcal A_\alpha^{\Phi}(\mathbb B^n)}$ is bounded,
then  $$T_g: \mathcal A_\alpha^{\Phi_p}(\mathbb B^n)\rightarrow \mathcal A_\alpha^{\Phi}(\mathbb B^n)$$ is compact if and only if for every sequence $(f_j)$ in
the unit ball of $\mathcal A_\alpha^{\Phi}(\mathbb B^n)$ which converges to 0 uniformly on compact subsets of $\mathbb B^n$,
one has   $$||T_g(f_j)||_{\Phi,\alpha}^{lux}\rightarrow 0\,\,\, \textrm{as}\,\,\, j\rightarrow \infty.$$

\end{lemma}
We have the following result.
\btheo\label{thm:compactcesaro}
Let $\Phi\in \mathscr{L}_p\cup \mathscr{U}^q$ and $\alpha>-1$. Assume $g$ is a holomorphic function on $\mathbb{B}^n$. Then the operator $T_g: \mathcal{A}_\alpha^\Phi(\mathbb{B}^n)\rightarrow \mathcal{A}_\alpha^\Phi(\mathbb{B}^n)$ is compact if and only if
 
\Be\label{eq:littleblochcondi}
\lim_{|z|\rightarrow 1}(1-|z|^2)|\mathcal{R}g(z)|=0.
\Ee

\etheo
\begin{proof}
Let us first assume that $T_g$ is compact. Let $\{a_j\}_{j\in \mathbb{N}}$ be a sequence in $\mathbb{B}^n$ such that $\lim_{j\rightarrow \infty}|a_j|=1$. Consider the sequence of holomorphic functions on $\mathbb{B}^n$ given by 
$$f_j(z)=\Phi^{-1}\left(\frac{1}{(1-|a_j|^2)^{n+1+\alpha}}\right)\left(\frac{1-|a_j|^2}{1-\langle z,a_j\rangle}\right)^{k(n+1+\alpha)}$$ 
with $k>1$ to be precised where needed. From Lemma \ref{lem:testfunctionberg}, we know that the sequence $\{f_j\}_{n\in \mathbb{N}}$ is uniformly bounded in $\mathcal{A}_\alpha^\Phi(\mathbb{B}^n)$. Also, we have that if $\Phi\in \mathscr{U}^q$, then as $\Phi^{-1}$ is concave and as $k>1$,
$$|f_j(z)|\le \frac{(1-|a_j|^2)^{(k-1)(n+1+\alpha)}}{|1-\langle z,a_j\rangle|^{k(n+1+\alpha)}}\rightarrow 0$$ 
as $j\rightarrow \infty$ on compact subsets of $\mathbb{B}^n$. If $\Phi\in \mathscr{L}_p$, then as $\Phi^{-1}\in \mathscr{U}^{\frac{1}{p}}$, taking $k>\frac{1}{p}$, we obtain using (\ref{uppertype}) that
$$|f_j(z)|\le C\frac{(1-|a_j|^2)^{(k-\frac{1}{p})(n+1+\alpha)}}{|1-\langle z,a_j\rangle|^{k(n+1+\alpha)}}\rightarrow 0$$
as $j\rightarrow \infty$ on compact subsets of $\mathbb{B}^n$.

Using Lemma \ref{lem:pointwiseforderivatives}, we obtain that there is a constant $C>0$ such that each $j\in \mathbb{N}$, and for any $z\in \mathbb{B}^n$,
$$ (1-|z|^2)|\mathcal{R}T_gf_j(z)| \le C\Phi^{-1}\left(\frac{1}{(1-|z|^2)^{n+1+\alpha}}\right)||T_g(f_j)||_{\Phi,\alpha}^{lux}$$
or equivalently, 
\Beas && (1-|z|^2)\Phi^{-1}\left(\frac{1}{(1-|a_j|^2)^{n+1+\alpha}}\right)\left|\frac{1-|a_j|^2}{1-\langle z,a_j\rangle}\right|^{k(n+1+\alpha)}|\mathcal{R}g(z)|\\ &\le& C\Phi^{-1}\left(\frac{1}{(1-|z|^2)^{n+1+\alpha}}\right)||T_g(f_j)||_{\Phi,\alpha}^{lux}.
\Eeas
Putting in particular $z=a_j$, we obtain
$$(1-|a_j|^2)|\mathcal{R}g(a_j)|\le C||T_g(f_j)||_{\Phi,\alpha}^{lux}.$$
As $||T_g(f_j)||_{\Phi,\alpha}^{lux}\rightarrow 0$ as $j\rightarrow \infty$, we deduce that 
$$\lim_{j\rightarrow \infty}(1-|a_j|^2)|\mathcal{R}g(a_j)|=0$$
which leads to 
$$\lim_{|z|\rightarrow 1}(1-|z|^2)|\mathcal{R}g(z)|=0.$$
Conversely, let us assume that the holomorphic function $g$ satisfies (\ref{eq:littleblochcondi}). Note that this implies that $g\in \mathcal{A}_\alpha^\Phi(\mathbb{B}^n)$ and that for any $\varepsilon>0 $, there exists $\eta$ such that 
\Be\label{eq:littleblochcut}
(1-|z|^2)|\mathcal{R}g(z)|<\varepsilon
\Ee
for any $z\in \mathbb{B}^n$ such that $\eta<|z|<1$. 

Let us start by proving that $T_g$ is bounded on $\mathcal{A}_\alpha^\Phi(\mathbb{B}^n)$. Let $$K=\max\{1,C\Phi^{-1}\left(\frac{1}{(1-\eta^2)^{n+1+\alpha}}\right),\|g\|_{\Phi,\alpha}^{lux}\}$$ where $C$ is (\ref{eq:pointwiseestimate}). We have at first that for any $f\in \mathcal{A}_\alpha^\Phi(\mathbb{B}^n)$,
\Beas
L &:=& \int_{\mathbb{B}^n}\Phi\left(\frac{(1-|z|^2)|\mathcal{R}T_gf(z)|}{K^2\|f\|_{\Phi,\alpha}^{lux}}\right)d\nu_\alpha(z)\\ &=& \int_{\mathbb{B}^n}\Phi\left(\frac{(1-|z|^2)|\mathcal{R}g(z)||f(z)|}{K^2\|f\|_{\Phi,\alpha}^{lux}}\right)d\nu_\alpha(z)\\ &\le& \int_{|z|\le \eta}\Phi\left(\frac{(1-|z|^2)|\mathcal{R}g(z)||f(z)|}{K^2\|f\|_{\Phi,\alpha}^{lux}}\right)d\nu_\alpha(z)+\\ & & \int_{|z|>\eta}\Phi\left(\frac{(1-|z|^2)|\mathcal{R}g(z)||f(z)|}{K^2\|f\|_{\Phi,\alpha}^{lux}}\right)d\nu_\alpha(z).
\Eeas
Using the pointwise estimate (\ref{eq:pointwiseestimate}) and the defintion of the constant $K$, we obtain
\Beas
L_1 &:=& \int_{|z|\le \eta}\Phi\left(\frac{(1-|z|^2)|\mathcal{R}g(z)||f(z)|}{K^2\|f\|_{\Phi,\alpha}^{lux}}\right)d\nu_\alpha(z)\\ &\le& \int_{|z|\le \eta}\Phi\left(\frac{(1-|z|^2)|\mathcal{R}g(z)|C\Phi^{-1}\left(\frac{1}{(1-|z|^2)^{n+1+\alpha}}\right)}{K^2}\right)d\nu_\alpha(z)\\ &\le& \int_{|z|\le \eta}\Phi\left(\frac{(1-|z|^2)|\mathcal{R}g(z)|C\Phi^{-1}\left(\frac{1}{(1-\eta^2)^{n+1+\alpha}}\right)}{K^2}\right)d\nu_\alpha(z).
\Eeas
It follows from the equivalent characterization in Theorem \ref{thm:main} that
\Beas
L_1 &\le& \int_{|z|\le \eta}\Phi\left(\frac{(1-|z|^2)|\mathcal{R}g(z)|}{\|g\|_{\Phi,\alpha}^{lux}}\right)d\nu_\alpha(z)\\ &\le& \int_{\mathbb{B}^n}\Phi\left(\frac{(1-|z|^2)|\mathcal{R}g(z)|}{\|g\|_{\Phi,\alpha}^{lux}}\right)d\nu_\alpha(z)\\ &\lesssim& \int_{\mathbb{B}^n}\Phi\left(\frac{|g(z)|}{\|g\|_{\Phi,\alpha}^{lux}}\right)d\nu_\alpha(z)\le 1.
\Eeas
That is
\Be\label{eq:L1}
\int_{|z|\le \eta}\Phi\left(\frac{(1-|z|^2)|\mathcal{R}g(z)||f(z)|}{K^2\|f\|_{\Phi,\alpha}^{lux}}\right)d\nu_\alpha(z)\lesssim 1.
\Ee
Using the estimate (\ref{eq:littleblochcut}), we obtain 
\Beas
L_2 &:=& \int_{|z|>\eta}\Phi\left(\frac{(1-|z|^2)|\mathcal{R}g(z)||f(z)|}{K^2\|f\|_{\Phi,\alpha}^{lux}}\right)d\nu_\alpha(z)\\ &\le& \int_{|z|>\eta}\Phi\left(\frac{\varepsilon |f(z)|}{K^2\|f\|_{\Phi,\alpha}^{lux}}\right)d\nu_\alpha(z)\\ &\le& \int_{\mathbb{B}^n}\Phi\left(\frac{|f(z)|}{\|f\|_{\Phi,\alpha}^{lux}}\right)d\nu_\alpha(z)\\ &\le& 1.
\Eeas
That is 
\Be\label{eq:L2}
\int_{|z|>\eta}\Phi\left(\frac{(1-|z|^2)|\mathcal{R}g(z)||f(z)|}{K^2\|f\|_{\Phi,\alpha}^{lux}}\right)d\nu_\alpha(z)\le 1.
\Ee
From (\ref{eq:L1}), (\ref{eq:L2}) and the equivalent characterizations in Theorem \ref{thm:main}, we deduce that $T_g$ is bounded on $\mathcal{A}_\alpha^\Phi(\mathbb{B}^n)$.
\vskip .2cm
Now, let $\{f_j\}_{j\in \mathbb{N}}$ be a sequence in the unit ball of $\mathcal{A}_\alpha^\Phi(\mathbb{B}^n)$ which converges to $0$ uniformly on compact subsets of $\mathbb{B}^n$. Then there exists an integer $j_0>0$ such that for any $j>j_0$,
$$\sup_{0<|z|\le \eta}|f_j(z)|<\varepsilon.$$
Let $M:=\max\{1,\|g\|_{\Phi,\alpha}^{lux}\}$. Then we obtain for any $j>j_0$, 
\Beas
L &:=& \int_{\mathbb{B}^n}\Phi\left(\frac{(1-|z|^2)|\mathcal{R}T_gf_j(z)|}{M}\right)d\nu_\alpha(z)\\ &=& \int_{\mathbb{B}^n}\Phi\left(\frac{(1-|z|^2)|\mathcal{R}g(z)||f_j(z)|}{M}\right)d\nu_\alpha(z)\\ &\le& \int_{|z|\le \eta}\Phi\left(\frac{(1-|z|^2)|\mathcal{R}g(z)||f_j(z)|}{M}\right)d\nu_\alpha(z)+\\ & & \int_{|z|>\eta}\Phi\left(\frac{(1-|z|^2)|\mathcal{R}g(z)||f_j(z)|}{M}\right)d\nu_\alpha(z).
\Eeas
Using the convexity of $\Phi$ if $\Phi\in \mathscr{U}^q$ and condition (\ref{eq:lowertype}) if $\Phi\in \mathscr{L}_p $ and the definition of $p_\Phi$  in (\ref{eq:pPhi}), it follows that
\Beas  L &\lesssim& \varepsilon^{p_\Phi}\int_{|z|\le \eta}\Phi\left(\frac{(1-|z|^2)|\mathcal{R}g(z)|}{M}\right)d\nu_\alpha(z)+\varepsilon^{p_\Phi}\int_{|z|>\eta}\Phi\left(\frac{|f_j(z)|}{M}\right)d\nu_\alpha(z)\\ &\lesssim& \varepsilon^{p_\Phi}\int_{\mathbb{B}^n}\Phi\left(\frac{|g(z)|}{M}\right)d\nu_\alpha(z)+\varepsilon^{p_\Phi}\int_{\mathbb{B}^n}\Phi\left(\frac{|f_j(z)|}{M}\right)d\nu_\alpha(z)\\ &\le& 2\varepsilon^{p_\Phi}.
\Eeas
It follows that $\int_{\mathbb{B}^n}\Phi\left(\frac{(1-|z|^2)|\mathcal{R}T_gf_j(z)|}{M}\right)d\nu_\alpha(z)\rightarrow 0$ as $j\rightarrow \infty$. Hence that $\int_{\mathbb{B}^n}\Phi\left(\frac{|T_gf_j(z)|}{M}\right)d\nu_\alpha(z)\rightarrow 0$ as $j\rightarrow \infty$. This implies that $\|T_gf_j\|_{\Phi,\alpha}^{lux}\rightarrow 0$ as $j\rightarrow \infty$. Thus $T_g$ is a compact operator on $\mathcal{A}_\alpha^\Phi(\mathbb{B}^n)$. The proof is complete.
\end{proof}

\bibliographystyle{plain}

\begin{thebibliography}{1}

\bibitem{Bekolle}
\textsc{B\'ekoll\'e, D.}: In\'egalit\'es \`a poids pour le projecteur de Bergman dans la boule unit\'e de $\mathbb C^n$. Studia Math. {\bf 71} (1981/82), no. 3, 305 � 323.

\bibitem{Berg}
\textsc{J. Bergh, J. L\"ofstr\"om}, Interpolation spaces, 1976.


\bibitem{BL}
\textsc{Bonami, A., Luo, L.}: On Hankel operators between
Bergman spaces on the unit ball. Houston J. Math. Vol. {\bf 31},
no. 3  (2005), 815 -- 828.



\bibitem{BS}
\textsc{Bonami, A.,  Sehba, B. F.}: Hankel operators between Hardy-Orlicz spaces and products of
holomorphic functions. Rev. Math. Arg. Vol. {\bf 50}, no. 2  (2009), 187 -- 199.


\bibitem{Cobos}
\textsc{F. Cobos, J. Peetre and L. E. Persson}, On the connection between real and complex interpolation of quasi-Banach spaces. Bull. Sci. Math. {\bf 122}  (1998), 17--37.
\bibitem{CM}
\textsc{Cowen, C. C. and MacCluer, B. D.}, Composition operators on spaces of analytic functions. Studies in Advanced Mathematics, CRC Press, Boca Katon. FL., 1995.

\bibitem{DHZZ}
\textsc{Y. Deng, L. Huang, T. Zhao, D. Zheng}: Bergman projection and Bergman spaces, J. Oper. Theor. {\bf 46} (2001), 3-24.

\bibitem{Faragallah}
\textsc{M. Faragallah}, Interpolation of weighted Orlicz spaces. Appl. Math. Comput. {\bf 145} (2003), 613--622.

\bibitem{Gustavsson1}
\textsc{J. Gustavsson}, On interpolation of weighted $L^p-$spaces and Ovchinnikov's theorem. Studia Math. {\bf 72} (1982), no. 3, 237-–251.


\bibitem{Gustavsson}
\textsc{J. Gustavsson, J. Peetre}, Interpolation of Orlicz spaces. Studia Math. {\bf 60} (1977), 33-–59.

\bibitem{Hu1}
\textsc{Hu, Z. J.}: Extended Cesaro operators
on the Bloch space in the unit ball of Cn, Acta Math. Sci. Ser. B Engl.
Ed. {\bf 23} (4) (2003), 561--566.

\bibitem{Hu2}
\textsc{Hu, Z. J.}: Extended Ces\`aro operators on mixed norm spaces. Proc.  Amer. Math. Soc. {\bf 131} (7), 2171--2179 (2003).


\bibitem{Hu3}
\textsc{Hu, Z. J.}: Extended Cesaro operators
on Bergman spaces, J. Math. Anal. Appl. {\bf 296} (2004), 435--454.








\bibitem{Kraynek}
\textsc{W. T. Kraynek}, Interpolation of sublinear operators in generalized Orlicz and Hardy-Orlicz spaces. Studi Math. {\bf 43} (1972), 93--123.

\bibitem{LLQR}
\textsc{Lef\`evre, P.,Li, D., Queff\'ellec, H. and Rodriguez-Piazza, L.}: Composition operators on Hardy-orlicz spaces, Memoirs of the AMS, {\bf 207}, no. 974 (2010).

\bibitem{LiStevic}
\textsc{Li, S, Stevic\', S}: Riemann-Stieltjes operators between different Bergman spaces. Bull. Belg. Math. Soc. Simon Stevin {\bf 15} (4) (2008), 677--686.


\bibitem{PauZhao2}
\textsc{J. Pau, R. Zhao}, Weak factorization and Hankel form for weighted Bergman spaces on the unit ball. Available as http://arxiv.org/abs/1407.4632v1.







\bibitem{Rao}
\textsc{M. M. Rao}, Interpolation, ergodicity and martingales. J. Math. Mech. {\bf 16} (1966), 543--567.

\bibitem{RR}
\textsc{M. M. Rao, Z. D. Ren,} Theory of Orlicz functions, Pure and Applied Mathematics {\bf 146}, Marcel Dekker, Inc. (1991).


\bibitem{R}
\textsc{W. Rudin,} Function theory in the unit ball of
$\C^n$. Grundlehren der Mathematischen Wissenschaften [Fundamental Principles of Mathematical Science], 241.
Springer-Verlag, New York-Berlin (1980).

\bibitem{sehbastevic}
\textsc{B. F. Sehba, S. Stevic,} On some product-type operators from
Hardy-Orlicz and Bergman-Orlicz spaces to weighted-type spaces,
Appl. Math. Comput. (2014)

\bibitem{ST}
\textsc{B. F. Sehba, E. Tchoundja,} Hankel operators with weighted Lipschitz symbols in the unit ball,  Math. Scand. {\bf 112}, (2013)  no. 2, 258-- 274.

\bibitem{ST1}
\textsc{B. F. Sehba, E. Tchoundja, } Hankel operators between holomorphic Hardy-Orlicz spaces,  Integral Equations Operator Theory {\bf 73} (2012), no. 3, 331--349.

\bibitem{ST2}
\textsc{B. F. Sehba, E. Tchoundja,} Duality for large Bergman-Orlicz spaces and bounded of Hankel operators.



\bibitem{V}
\textsc{Viviani, B. E.}: An atomic decomposition of the predual of $BMO(\rho)$.
Rev. Mat. Iberoamericana {\bf 3} (1987), no. 3-4, 401 -- 425.

\bibitem{VT}
\textsc{Volberg, A. L. and Tolokonnikov, V. A.}: Hankel operators and problems of best approximation of unbounded functions.
Translated from Zapiski Nauchnykh Seminarov Leningradskogo Matematicheskogo Instituta im. V.A. Steklova AN SSSR,  {\bf 141} (1985), 5 -- 17.

\bibitem{Xiao}
\textsc{Xiao, J. }, Riemann-Stieltjes operators on weighted Bloch and Bergman spaces of the unit ball. J. London Math. Soc. (2) {\bf 70} (2004),no. 1, 199--214.

\bibitem{ZZ}
\textsc{Zhao, R. and Zhu, K.}: Theory of Bergman spaces in the unit ball of $\Bbb C^n$Cn. Mem. Soc. Math. Fr. (N.S.) No. {\bf 115} (2008), vi+103 pp. (2009). ISBN: 978-2-85629-267-9.


\bibitem{KZ} \textsc{Zhu, K.}:
\newblock{Spaces of holomorphic functions in the unit ball.} Graduate Texts in Mathematics 226, Springer  Verlag (2004).

\end{thebibliography}

\end{document}